\documentclass[10pt]{amsart}
\usepackage{amsmath}
\usepackage[usenames,dvipsnames]{color}
\usepackage{parskip}
\usepackage{amsfonts}
\usepackage{amscd}
\usepackage[centertags]{amsmath}
\usepackage{amssymb}
\usepackage{stmaryrd}
\usepackage[all,cmtip]{xy}
\usepackage[english]{babel}
\usepackage{tabularx}
\usepackage{amsxtra}
\usepackage{euscript}
\usepackage[T1]{fontenc}
\usepackage{doc, exscale, fontenc, latexsym, syntonly}
\usepackage{amsfonts}
\usepackage{amsthm}
\usepackage{graphicx}
\usepackage{mciteplus}
\usepackage{cite}

\newtheorem{theorem}{Theorem}[section]
\newtheorem{corollary}[theorem]{Corollary}
\newtheorem{lemma}[theorem]{Lemma}
\newtheorem{proposition}[theorem]{Proposition}

\theoremstyle{definition}
\newtheorem{definition}[theorem]{Definition}
\newtheorem{example}[theorem]{Example}
\theoremstyle{remark}

\numberwithin{figure}{section}
\numberwithin{table}{section}

\begin{document}

\title[Discrete Topological Complexities of Simplical Maps]{Discrete Topological Complexities of Simplicial Maps}

\author{MEL\.{I}H \.{I}S}
\date{\today}

\address{\textsc{Melih Is,}
Ege University\\
Faculty of Sciences\\
Department of Mathematics\\
Izmir, Turkiye}
\email{melih.is@ege.edu.tr}

\author{\.{I}SMET KARACA}
\date{\today}

\address{\textsc{Ismet Karaca,}
	Ege University\\
	Faculty of Sciences\\
	Department of Mathematics\\
	Izmir, Turkiye}
\email{ismet.karaca@ege.edu.tr}



\begin{abstract}
 In this study, we delve into the discrete TC of surjective simplicial fibrations, aiming to unravel the interplay between topological complexity, discrete geometric structures, and computational efficiency. Moreover, we examine the properties of the discrete TC number in higher dimensions and its relationship with scat. We also touch on the basic properties of the notion of higher contiguity distance, and show that it is possible to consider discrete TC computations in a simpler sense.
\end{abstract}

\maketitle

\section{Introduction}
\label{intro}

\quad The discrete topological complexity (TC) of a space serves as a fundamental measure of capturing the intricacy of its motion-planning capabilities. Originating from the field of robotics, TC offers a quantitative framework to understand the computational complexity of designing feasible paths in a given space. Particularly, in algebraic topology, TC provides valuable insights into the structural characteristics of topological spaces and their associated mappings.

\quad The notion of the discrete topological complexity on simplicial complexes is first given in \cite{TerMacMinVil:2018} by using Farber subcomplexes. Theorem 3.4 of \cite{TerMacMinVil:2018} relates this characterization to contiguity distance, which is the discrete version of the concept of homotopic distance. The contiguity distance between simplicial maps is studied in \cite{BoratPamukVergili:2023}, and hence, some homotopy-related concepts, such as contractibility or having the same homotopy type, are transferred from topological spaces to simplicial complexes. With the help of these studies, it has now become possible to examine the problem of determining the TC number of a simplicial map via the contiguity distance. On the other hand, a fibration between simplicial complexes is introduced in \cite{TerCalMacVil:2021}. Moreover, in Theorem 8 of \cite{TerCalMacVil:2021}, the discrete TC number of a finite simplicial complex $L$ is presented by using the simplicial path-fibration $$PL \rightarrow L \times L.$$ In this study, we focus on investigating the TC of surjective simplicial fibrations (generally between finite complexes), a class of mappings that exhibit crucial properties in both algebraic topology and differential geometry. Surjective simplicial fibrations serve as essential tools for studying the topology of fiber bundles, providing a means to understand the interplay between base spaces and fibers. Our exploration of TC within this context aims to shed light on the computational complexity underlying the continuous deformation of spaces under surjective simplicial fibrations.

Understanding TC in the context of surjective simplicial fibrations entails a comprehensive analysis of discrete structures that underlie continuous mappings. By discretizing the domain and codomain of such mappings, we can effectively capture the essential geometric and topological features while providing a computationally tractable framework for analysis. Through this, we aim to unravel the intricate interplay between the topological complexity of the base space and the geometric properties of the fiber, elucidating how these factors collectively influence the TC of surjective simplicial fibrations. 

\quad This exploration consists of the following concepts: In Section \ref{sec:1}, we recall the basic properties of simplicial complexes and the important consequences of maps between simplicial complexes, especially simplicial fibrations. In Section \ref{sec:2}, we present the discrete topological complexity of a surjective fibration via the Schwarz genus of a simplicial fibration. This definition is enriched with different types of examples of simplicial complexes. We also generalize the notion of contiguity distance to use it effectively in other sections. The following two sections, Section \ref{sec:3} and \ref{sec:4}, deal with the generalized version of TC number computation of a simplicial complex and a surjective simplicial fibration. Furthermore, Section \ref{sec:5} is dedicated to the study of the relationship, in the discrete sense, between TC numbers and the Lusternik-Schnielmann category of simplicial complexes denoted by scat. 

\section{Preliminaries}
\label{sec:1}

\quad Simplicial complexes are fundamental structures in algebraic topology, providing a combinatorial framework for studying topological spaces. They are constructed from simple geometric elements called simplices, which are higher-dimensional analogs of triangles and tetrahedra. We now present the general properties of simplicial complexes or maps between them. 

\subsection{Simplicial Complex and Simplicial Homotopy}
\label{subsec:1}

A \textit{simplicial complex} $L$ is a set of simplexes in $\mathbb{R}^{n}$ which satifies
\begin{itemize}
	\item $\sigma \in L$ implies that $L$ has every face of $\sigma$, and
	\item $\sigma_{1}$, $\sigma_{2} \in L$ implies that the intersection $\sigma_{1} \cap \sigma_{2}$ is equal to either null or a common face of $\sigma_{1}$ and of $\sigma_{2}$ \cite{Rotman:2013}.
\end{itemize}

\quad If $L$ has a finite collection of simplexes that satisfies the above conditions, then we say that $L$ is a \textit{finite simplicial complex}. The \textit{vertex set} of a simplicial complex $L$ is defined by the collection of all points ($0-$simplexes) in $L$, and we denote it by VX$(L)$. Let $N$ and $L$ be any simplicial complexes. Then $N$ is called a \textit{subcomplex} of $L$ if $\sigma \in N$, then $\sigma \in L$ with the property VX$(N) \subset$ VX$(L)$ \cite{Rotman:2013}. 

\begin{definition}\cite{Rotman:2013}
	A map $\varphi : L \rightarrow L^{'}$ between any simplicial complexes $L$ and $L^{'}$ is called a \textit{simplicial map} provided that the map $\varphi : \text{VX}(L) \rightarrow \text{VX}(L^{'})$ has the property that $\sigma \in L$ implies $\varphi(\sigma) \in L^{'}$.
\end{definition}

\quad A simplicial map $\varphi : L \rightarrow L^{'}$ is called a \textit{simplicial isomorphism} if it is bijective, and the inverse $\varphi^{-1}$ is a simplicial map.

\quad Given two simplicial maps $\varphi_{1}$, $\varphi_{2} : L \rightarrow L^{'}$, they are said to be \textit{contiguous} provided that the fact $\sigma \in L$ is a simplex implies that $\varphi_{1}(\sigma) \cup \varphi_{2}(\sigma) \in L^{'}$ is also a simplex \cite{Spanier:1966}. The contiguity of two simplicial maps $\varphi_{1}$ and $\varphi_{2}$ are generally denoted by $\varphi_{1} \sim_{c} \varphi_{2}$. For simplicial maps, the contiguity is known as the homotopy counterpart and is defined so that various simplicial approximations to the same continuous map are contiguous. Note that being in the same contiguity class for simplicial complexes and simplicial maps can be thought of as the discrete form of homotopy.

\begin{definition}\cite{Spanier:1966}
	Given two simplicial maps $\varphi$, $\varphi^{'} : L \rightarrow L^{'}$, they are \textit{in the same contiguity class with $n$ steps} provided that there exists a sequence of simplical maps $\varphi_{i} : L \rightarrow L^{'}$ for $i = 0,\cdots,n$ that satisfes $\varphi_{i} \sim_{c} \varphi_{i+1}$ with $\varphi_{0} = \varphi$ and $\varphi_{n} = \varphi^{'}$.
\end{definition}

\quad The notation $\varphi \sim \varphi^{'}$ is generally used to express that two simplicial maps $\varphi$ and $\varphi^{'}$ are in the same contiguity class.

\begin{proposition}\cite{TerCalMacVil:2021}
	Assume that $\varphi$, $\varphi^{'} : L \rightarrow L^{'}$ are two simplicial maps. Then $\varphi \sim \varphi^{'}$ if and only if there is at least one $m \geq 1$ and one simplicial map $$G : L \times I_{m} \rightarrow L^{'}$$ with the property $G(\sigma,0) = \varphi$ and $G(\sigma,m) = \varphi^{'}$ for any $\sigma \in L$.
\end{proposition}

\quad Assume that $L$ and $L^{'}$ are two simplicial complexes. Then their \textit{categorical product} $L \ \Pi \ L^{'}$ is defined as follows \cite{Kozlov:2008}:
\begin{itemize}
	\item For any vertex $v_{1} \in L$ and $v_{2} \in L^{'}$, the vertices of $L \ \Pi \ L^{'}$ are the pairs $(v_{1},v_{2})$, i.e., $$\text{VX}(L \ \Pi \ L^{'}) = \text{VX}(L) \times \text{VX}(L^{'}).$$
	\item For the projections $\pi_{1} : L \ \Pi \ L^{'} \rightarrow L$, $\pi_{2} : L \ \Pi \ L^{'} \rightarrow L^{'}$, we have that $\sigma \in L \ \Pi \ L^{'}$ if and only if $\pi_{1}(\sigma) \in L$ and $\pi_{2}(\sigma) \in L^{'}$.
\end{itemize}

\quad We use the notation $K \times L$ for the categorical product of simplicial complexes throughout the paper. For instance, $L^{2} = L \times L = L \ \Pi \ L$. 

\quad Strong homotopy type and contractibility for topological spaces are transferred to simplicial complexes as follows. Let $L$ and $N$ be two simplicial complexes. Then they \textit{have the same strong homotopy type} if and only if there exist two simplicial maps $\varphi : L \rightarrow N$ and $\omega : N \rightarrow L$ with $\varphi \circ \omega \sim 1_{N}$ and $\omega \circ \varphi \sim 1_{L}$ \cite{BarmakMinian:2012}. Also, $\varphi$ and $\omega $ are called the \textit{strong equivalences}. Let $v$ be a vertex in a simplicial complex $L$. Then $L$ is called \textit{strongly collapsible} if $L$ and $v$ have the same strong homotopy type.

\subsection{Simplicial Fibration and Discrete TC Number}
\label{subsec:2}

In \cite{TerCalMacVil:2021}, we have three equivalent definitions of the notion of a simplicial fibration. Since we would like to compute the discrete topological complexity of simplicial maps (actually surjective fibrations), it is essential to define a simplicial fibration. We prefer to use Type III of \cite{TerCalMacVil:2021} because it is almost the same as the fibrations defined with the help of homotopy in topological spaces.

\begin{definition}\cite{TerCalMacVil:2021}\label{def3}
	Let $\varphi : L \rightarrow L^{'}$ be a simplicial map. Then $\varphi$ is said to be a \textit{simplicial fibration} if for an inclusion map $i^{m} : N \times \{0\} \rightarrow N \times I_{m}$, and any simplicial maps $g : N \times \{0\} \rightarrow L$ and $G : N \times I_{m} \rightarrow L^{'}$ with $\varphi \circ g = G \circ i^{m}$, there exists a simplicial map $$\widetilde{G} : N \times I_{m} \rightarrow L$$ for which $\widetilde{G} \circ i_{m} = g$ and $\varphi \circ \widetilde{G} = G$.
\end{definition}

\quad In a special case in Definition \ref{def3}, if $N$ is finite, then $\varphi$ is called a \textit{simplicial finite-fibration}. Simplicial fibrations have some important properties. For example, any simplicial isomorphism is a simplicial fibration. Moreover, each of the composition, the pullback, and the cartesian product of simplicial fibrations is again a simplicial fibration \cite{TerCalMacVil:2021}. Another important example given by Theorem 1 of \cite{TerCalMacVil:2021}:

\begin{theorem}\cite{TerCalMacVil:2021}
	For any simplicial complex $L$, the map $\pi : PL \rightarrow L \times L$, defined by taking any simplicial path on $L$ to the pair of initial-desired vertices of $L$, is a simplicial finite-fibration.
\end{theorem}

\quad The simplicial Schwarz genus and the contiguity distance are two different ways to state the discrete TC of a simplicial complex when we have simplicial fibrations. Hence, we now continue with presenting these two concepts.

\begin{definition}\cite{TerCalMacVil:2021}
	Let $\varphi : L \rightarrow L^{'}$ be a simplicial map. Then the simplicial Schwarz genus of $\varphi$ is the least integer $n \geq 0$ if the following properties hold:
	\begin{itemize}
		\item $L^{'}$ can be written as the union of subcomplexes $L_{0},L_{1},\cdots,L_{n}$.
		\item For each $k \in \{0,\cdots,n\}$, $\varphi$ admits a simplical map $\sigma_{k} : L_{k} \rightarrow L$ with the property $\varphi \circ \sigma_{k} = 1_{L_{k}}$.
	\end{itemize} 
\end{definition}

\quad The simplicial Schwarz genus of $\varphi$ is denoted by Sg$(\varphi)$.

\begin{definition}\cite{BoratPamukVergili:2023}
	Let $\varphi_{1}$, $\varphi_{2} : L \rightarrow L^{'}$ be two simplicial maps. Then the contiguity distance between $\varphi_{1}$ and $\varphi_{2}$ is the least integer $n \geq 0$ if the following properties hold:
		\begin{itemize}
		\item $L$ can be written as the union of subcomplexes $L_{0},L_{1},\cdots,L_{n}$.
		\item For all $k \in \{0,\cdots,n\}$, $\varphi_{1}\big|_{L_{k}}$ and $\varphi_{2}\big|_{L_{k}}$ are in the same contiguity class.
	\end{itemize} 
\end{definition}

\quad The contiguity distance between $\varphi_{1}$ and $\varphi_{2}$ is denoted by SD$(\varphi_{1},\varphi_{2})$.

\quad We are now ready to give the discrete TC number and the simplicial Lusternik-Schnirelmann category based on the contiguity distance as follows.

\begin{proposition}\cite{BoratPamukVergili:2023}
	Let $p_{i}$ be the $i-$th simplicial projection map on $L$ for each $i \in \{1,2\}$, and $c_{v_{0}}$ any simplicial constant map on $L$, where $v_{0}$ is any vertex of $L$. Assume that $i_{1} : L \rightarrow L^{2}$ and $i_{2} : L \rightarrow L^{2}$ are simplicial maps defined by $i_{1}(\sigma) = (\sigma,v_{0})$ and $i_{2}(\sigma) = (v_{0},\sigma)$, respectively. Then
	
	\textbf{i)} TC$(L) =$ SD$(p_{1},p_{2})$.
	
	\textbf{ii)} scat$(L) =$ SD$(1_{L},c_{v_{0}}) =$ SD$(i_{1},i_{2})$.
	
	\textbf{iii)} scat$(\varphi) =$ SD$(\varphi,\varphi \circ c_{v_{0}})$.
\end{proposition}

\quad In computations of TC and scat, we always assume that a given simplicial complex is edge-path connected to make them considerable.

\section{Schwarz Genus Form and Higher Contiguity Distance}
\label{sec:2}

\quad For a surjective simplicial map $\varphi : L \rightarrow L^{'}$ between any finite simplicial complexes $L$ and $L^{'}$, define a new surjective simplicial map $$\pi_{\varphi} : L^{I} \rightarrow L \times L^{'}$$ by $\pi_{\varphi}(\delta) = ((1_{L} \times \varphi) \circ \pi)(\delta) = (\delta(0),\varphi(\delta(1))) $ for all $\varphi \in L^{I}$. Assume that $\varphi$ is a simplicial fibration. Then, by using Proposition 4.4 and 4.1 iii) of \cite{TerCalMacVil:2021}, $\pi_{\varphi}$ is also a simplicial fibration.

\begin{definition}
	The discrete topological complexity TC$(\varphi)$ of a simplicial finite-fibration $\varphi : L \rightarrow L^{'}$ is Sg$(\pi_{\varphi})$.
\end{definition}

\begin{example}
	\textbf{i)} TC$(\varphi)$ generalizes TC$(L)$. Indeed, for the particular case of $\varphi = 1_{L}$, we observe that TC$(1_{L}) =$ TC$(L)$.
	
	\textbf{ii)} The discrete topological complexity of a constant simplicial fibration is null, i.e., TC$(\varphi : L \rightarrow \{s_{0}\}) = 0$, where $s_{0}$ is a $0-$simplex (see Example 3.2 of \cite{iskaraca:2021} for a similar construction in digital images). Note that TC$(\varphi^{'})$ cannot be computed for $\varphi^{'} : L \rightarrow L^{'}$, defined by $\varphi^{'}(\delta) = \{s_{0}\} \in L^{'}$, because $\varphi^{'}$ is not surjective.
	
	\textbf{iii)} The discrete topological complexity of a first projection map is null, i.e., TC$(\varphi_{pr_{1}} : L \times L^{'} \rightarrow L) = 0$ (see Example 3.3 of \cite{iskaraca:2021} for a similar construction in digital images).
\end{example}

\begin{figure}[h]
	\centering
	\includegraphics[width=0.95\textwidth]{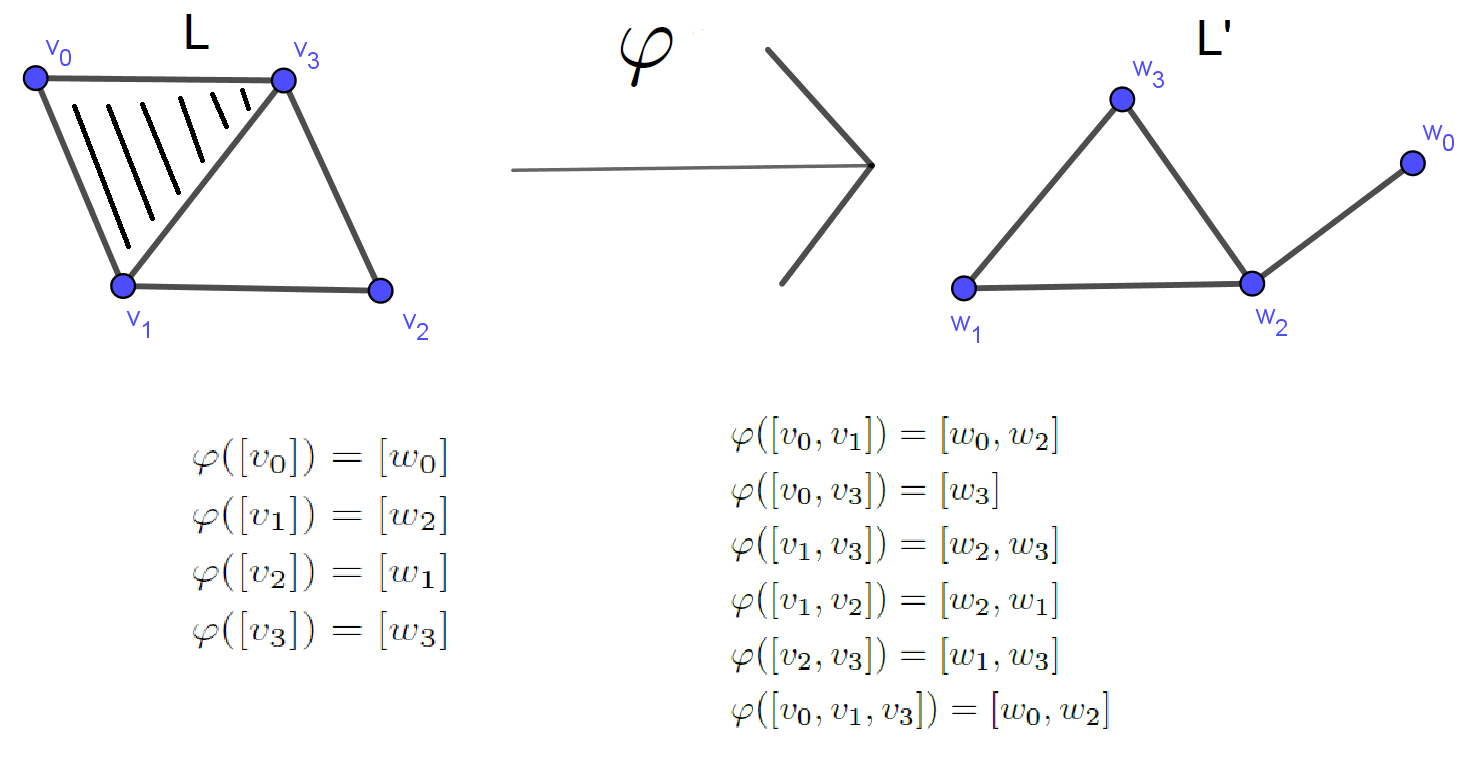}
	\caption{A simplicial map $\varphi : L \rightarrow L^{'}$.}
	\label{fig:1}
\end{figure}

\begin{example}
	Consider a simplicial map $\varphi : L \rightarrow L^{'}$ defined in Figure \ref{fig:1}. Obviously, it is a bijective simplicial map. If one defines the inverse of $\varphi$ from $L^{'}$ to $L$ as
	\begin{eqnarray*}
		&&[w_{0}] \mapsto [v_{0}]\\
		&&[w_{1}] \mapsto [v_{2}]\\
		&&[w_{2}] \mapsto [v_{1}]\\
	    &&[w_{3}] \mapsto [v_{3}],\\	
	\end{eqnarray*}
    and
	\begin{eqnarray*}
		&&[w_{0},w_{2}] \mapsto [v_{0},v_{1}]\\
		&&[w_{1},w_{2}] \mapsto [v_{2},v_{1}]\\
		&&[w_{1},w_{3}] \mapsto [v_{2},v_{3}]\\
		&&[w_{2},w_{3}] \mapsto [v_{1},v_{3}],\\	
	\end{eqnarray*}
	then $\varphi$ is a simplicial isomorphism. By Proposition 4 i) of \cite{TerCalMacVil:2021}, this concludes that $\varphi$ is a simplicial fibration. Define a simplicial fibration $\pi_{\varphi} : L^{I} \rightarrow L \times L^{'}$ by $\pi_{\varphi}(\delta) = (\delta(0),\varphi(\delta(1)))$. The set $L^{'}$ can be written as the union of $L_{0}$ and $L_{1}$ as in Figure \ref{fig:2}. Therefore, we get
	\begin{eqnarray*}
		L \times L^{'} = (L \times L_{0}) \cup (L \times L_{1}).
	\end{eqnarray*}
    In addition, $\pi_{\varphi}$ admits two simplicial maps $\sigma_{1} : L \times L_{0} \rightarrow L^{I}$ and $\sigma_{2} : L \times L_{1} \rightarrow L^{I}$ defined by $\sigma_{1}([a],[b]) = \alpha$ and $\sigma_{2}([c],[d]) = \beta$, respectively, with the property $\pi_{\varphi} \circ \sigma_{1}$ and $\pi_{\varphi} \circ \sigma_{2}$ is the inclusion map on $L \times L_{i}$ for each $i = 0,1$, where $\alpha$ is a simplicial path from $a$ to $b$ in $L$ and $\beta$ is a simplicial path from $c$ to $d$ in $L$. Consequently, we obtain TC$(\varphi) = 2$.
\end{example}

\begin{figure}[h]
	\centering
	\includegraphics[width=0.80\textwidth]{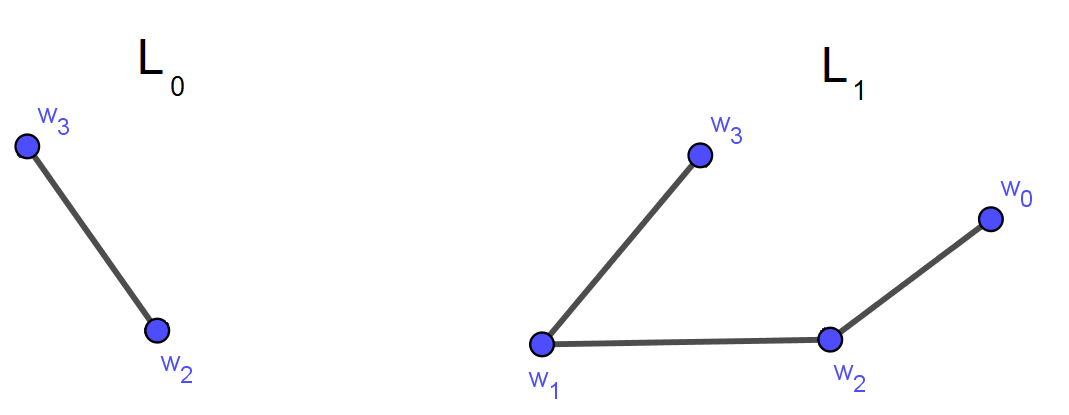}
	\caption{The subcomplexes $L_{0}$ and $L_{1}$ of $L^{'}$.}
	\label{fig:2}
\end{figure}

\quad Similar to the homotopic distance between maps, the notion of contiguity distance between simplicial complexes can be generalized as a higher contiguity distance between simplicial complexes.  

\begin{definition}\label{def1}
	Let $\varphi_{1},\cdots,\varphi_{m} : L \rightarrow L^{'}$ be simplicial maps. Then the higher ($n-$th) contiguity distance SD$(\varphi_{1},\cdots,\varphi_{m})$ is the least integer $n \geq 0$ for which there is a set of subcomplexes $L_{0},L_{1},\cdots,L_{n}$ that covers $L$ with the property that $\varphi_{i}|_{L_{k}}$ and $\varphi_{j}|_{L_{k}}$ are in the same contiguity class for all $i,j \in \{1,\cdots,m\}$ and $k = 0,1,\cdots,n$.
\end{definition}

\quad We have some quick observations from Definition \ref{def1}. First one states that the order of simplicial maps $\varphi_{1},\cdots,\varphi_{m}$ does not change the result of SD$(\varphi_{1},\cdots,\varphi_{m})$. More precisely, for any permutation $\sigma$ of $\{1,\cdots,m\}$, we have that
\begin{eqnarray*}
	\text{SD}(\varphi_{1},\cdots,\varphi_{m}) = \text{SD}(\varphi_{\sigma_{(1)}},\cdots,\varphi_{\sigma_{(m)}}).
\end{eqnarray*}

Second, by letting $1 < m^{'} < m$, we observe that 
\begin{eqnarray*}
	\text{SD}(\varphi_{1},\cdots,\varphi_{m^{'}}) \leq \text{SD}(\varphi_{1},\cdots,\varphi_{m})
\end{eqnarray*} 
for any simplicial maps $\varphi_{1},\cdots,\varphi_{m} : L \rightarrow L^{'}$. Moreover, we have that
\begin{eqnarray*}
	\text{SD}(\varphi_{1},\cdots,\varphi_{m}) = 0
\end{eqnarray*} 
iff $\varphi_{i} \sim \varphi_{i+1}$ for each $i \in \{1,\cdots,m\}$.

\quad The following properties of the higher contiguity distance are generalizations of the properties in \cite{BoratPamukVergili:2023} and, in parallel, proofs can be easily obtained using the same methods in \cite{BoratPamukVergili:2023}.

\begin{proposition}\label{prop2}
	\textbf{i)} Let $\varphi_{i},\varphi^{'}_{i} : L \rightarrow L^{'}$ be simplicial maps for all $i = 1,\cdots,m$. If $\varphi_{i} \sim \varphi^{'}_{i}$ for each $i$, then SD$(\varphi_{1},\cdots,\varphi_{m}) =$ SD$(\varphi^{'}_{1},\cdots,\varphi^{'}_{m})$.
	
	\textbf{ii)} Let $\varphi_{i} : L \rightarrow L^{'}$ be any simplicial maps for all $i = 1,\cdots,m$. If $L$ or $L^{'}$ is strongly collapsible, then SD$(\varphi_{1},\cdots,\varphi_{m}) = 0$.
\end{proposition}

\begin{lemma}\label{lem1}
	Let $\psi_{i} \sim \psi_{i+1}$ for any $i = 1,\cdots,m$. Assume that $\psi_{i}$ admits a simplicial map $\mu_{i}$ such that $\mu_{i} \circ \psi_{i} \sim 1$ (or $\psi_{i} \circ \mu_{i} \sim 1$) for each $i = 1,\cdots,m+1$. Then $\mu_{i} \sim \mu_{i+1}$ for all $i = 1,\cdots,m$.
\end{lemma}

\begin{proof}
	Suppose that $\mu_{i} \nsim \mu_{i+1}$ for all $i = 1,\cdots,m$. Then $\mu_{i} \circ \psi_{i+1} \nsim \mu_{i+1} \circ \psi_{i+1}$. Since $\psi_{i} \sim \psi_{i+1}$ for any $i = 1,\cdots,m$, we get $1 \sim \mu_{i} \circ \psi_{i} \nsim \mu_{i+1} \circ \psi_{i+1} \sim 1$. This is a contradiction.
\end{proof}

\begin{proposition}\label{prop1}
	\textbf{i)} Let $\varphi_{i} : L \rightarrow L^{'}$ and $\psi_{i} : L^{'} \rightarrow L^{''}$ be any simplicial maps for all $i = 1,\cdots,m$. If $\psi_{i} \sim \psi_{i+1}$ for all $i = 1,\cdots,m-1$, then $$\text{SD}(\psi_{1} \circ \varphi_{1},\cdots,\psi_{m} \circ \varphi_{m}) \leq \text{SD}(\varphi_{1},\cdots,\varphi_{m}).$$ Moreover, the equality holds provided that, for all $i = 1,\cdots,m$, $\psi_{i}$ admits a simplicial map $\mu_{i} : L^{''} \rightarrow L^{'}$ satisfying $\mu_{i} \circ \psi_{i} \sim 1_{L^{'}}$, and $\psi_{i} \circ \varphi_{i} \sim \psi_{j} \circ \varphi_{j}$ for any distinct $i$, $j = 1,\cdots,m$.
	
	\textbf{ii)} Let $\varphi_{i} : L \rightarrow L^{'}$ and $\psi_{i} : L^{''} \rightarrow L$ be any simplicial maps for all $i = 1,\cdots,m$. If $\psi_{i} \sim \psi_{i+1}$ for all $i = 1,\cdots,m-1$, then $$\text{SD}(\varphi_{1} \circ \psi_{1},\cdots,\varphi_{m} \circ \psi_{m}) \leq \text{SD}(\varphi_{1},\cdots,\varphi_{m}).$$ Moreover, the equality holds provided that, for all $i = 1,\cdots,m$, $\psi_{i}$ admits a simplicial map $\mu_{i} : L \rightarrow L^{''}$ satisfying $\psi_{i} \circ \mu_{i} \sim 1_{L}$, and $\varphi_{i} \circ \psi_{i} \sim \varphi_{j} \circ \psi_{j}$ for any distinct $i$, $j = 1,\cdots,m$.
\end{proposition}

\begin{proof}
	Let SD$(\varphi_{1},\cdots,\varphi_{m}) = n$. Then there is a set of subcomplexes $L_{0},L_{1},\cdots,L_{n}$ that covers $L$ with the property that $\varphi_{i}|_{L_{k}}$ and $\varphi_{j}|_{L_{k}}$ are in the same contiguity class for all $i,j \in \{1,\cdots,m\}$ and $k = 0,1,\cdots,n$, i.e., $\varphi_{1}|_{L_{k}} \sim \cdots \sim \varphi_{m}|_{L_{k}}$.
	
	\textbf{i)} We get
	\begin{eqnarray*}
		\big(\psi_{s} \circ \varphi_{s}\big)|_{L_{k}} = \psi_{s} \circ \varphi_{s}\big|_{L_{k}} \sim \psi_{t} \circ \varphi_{t}\big|_{L_{k}} = \big(\psi_{t} \circ \varphi_{t}\big)|_{L_{k}}
	\end{eqnarray*}
for any $s$, $t = 1,\cdots,m$ with $s \neq t$. This shows that $SD(\psi_{1} \circ \varphi_{1},\cdots,\psi_{m} \circ \varphi_{m}) \leq n$. In addition, by assuming that there exists a simplicial map $\mu_{i} : L^{''} \rightarrow L^{'}$ with $\mu_{i} \circ \psi_{i} \sim 1_{L^{'}}$, and $\psi_{i} \circ \varphi_{i} \sim \psi_{j} \circ \varphi_{j}$ for any distinct $i$, $j = 1,\cdots,m$, we get
\begin{eqnarray*}
	\text{SD}(\varphi_{1},\cdots,\varphi_{m}) &=& \text{SD}(\mu_{1} \circ \psi_{1} \circ \varphi_{1},\cdots,\mu_{m} \circ \psi_{m} \circ \varphi_{m})\\
	&\leq& \text{SD}(\psi_{1} \circ \varphi_{1},\cdots,\psi_{m} \circ \varphi_{m})\\
	&\leq& \text{SD}(\varphi_{1},\cdots,\varphi_{m})
\end{eqnarray*}
from Lemma \ref{lem1}.
	
	\textbf{ii)} For any $L_{k} \subseteq L$, $k = 0,1,\cdots,n$, we consider $L^{''}_{k} = \psi_{i}^{-1}(L_{k}) \subseteq L^{''}$. Then $L^{''}_{0},L^{''}_{1},\cdots,L^{''}_{n}$ are subcomplexes that cover $L^{''}$. Moreover, by assuming that the map $\omega_{k,i} : L^{''}_{k} \rightarrow L_{k}$ is the restriction of $\psi_{i}$, we get
	\begin{eqnarray*}
	\big(\varphi_{s} \circ \psi_{s}\big)|_{L^{''}_{k}} &=& \varphi_{s}\big|_{L^{''}_{k}} \circ \omega_{k,s} \sim \varphi_{t}\big|_{L^{''}_{k}} \circ \omega_{k,t}\\
	&=& \varphi_{t} \circ \text{incl}_{L_{k}} \circ \omega_{k,t}\\
	&=& \varphi_{t} \circ \psi_{t}\big|_{L^{''}_{k}} = \big(\psi_{t} \circ \varphi_{t}\big)|_{L^{''}_{k}}
    \end{eqnarray*}
    for any $s$, $t = 1,\cdots,m$ with $s \neq t$ and the inclusion map incl$_{L_{k}} : L_{k} \rightarrow L$. This shows that $SD(\psi_{1} \circ \varphi_{1},\cdots,\psi_{m} \circ \varphi_{m}) \leq n$. In addition, by assuming that there exists a simplicial map $\mu_{i} : L \rightarrow L^{''}$ with $\psi_{i} \circ \mu_{i} \sim 1_{L}$, and $\varphi_{i} \circ \psi_{i} \sim \varphi_{j} \circ \psi_{j}$ for any distinct $i$, $j = 1,\cdots,m$ we get
     \begin{eqnarray*}
      \text{SD}(\varphi_{1},\cdots,\varphi_{m}) &=& \text{SD}(\varphi_{1} \circ \psi_{1} \circ \mu_{1},\cdots,\varphi_{m} \circ \psi_{m} \circ \mu_{m})\\
      &\leq& \text{SD}(\varphi_{1} \circ \psi_{1},\cdots,\varphi_{m} \circ \psi_{m})\\
      &\leq& \text{SD}(\varphi_{1},\cdots,\varphi_{m})
      \end{eqnarray*}
      from Lemma \ref{lem1}.
\end{proof}

\begin{corollary}\label{cor1}
	Let $\varphi_{1},\cdots,\varphi_{m} : L \rightarrow L^{'}$ and $\psi_{1},\cdots,\psi_{m} : N \rightarrow N^{'}$ be simplicial maps. Assume that $\beta : N \rightarrow L$ and $\alpha : L^{'} \rightarrow N^{'}$ have the same strong homotopy type and the diagram
	$$\xymatrix{
		L \ar[r]^{\varphi_{1},\cdots,\varphi_{m}} &
		L^{'} \ar[d]^{\alpha} \\
		N \ar[r]_{\psi_{1},\cdots,\psi_{m}} \ar[u]^{\beta} & N^{'}}$$
	commutes with respect to the contiguity (in other words, $\alpha \circ \varphi_{i} \circ \beta \sim \psi_{i}$ for all $i = 1,\cdots,m$). Then SD$(\varphi_{1},\cdots,\varphi_{m}) =$ SD$(\psi_{1},\cdots,\psi_{m})$.
\end{corollary}

\begin{proof}
	By using Proposition \ref{prop2} i), and Proposition \ref{prop1} i) and ii), respectively, we find
	\begin{eqnarray*}
		\text{SD}(\psi_{1},\cdots,\psi_{m}) = \text{SD}(\alpha \circ \varphi_{1} \circ \beta,\cdots,\alpha \circ \varphi_{m} \circ \beta) = \text{SD}(\varphi_{1},\cdots,\varphi_{m}).
	\end{eqnarray*}
\end{proof}

\section{Contiguity Distance Form}
\label{sec:3}

\quad We know that the discrete topological complexity TC$(L)$ can be expressed by the contiguity distance of two projection maps, i.e., TC$(L) =$ SD$(p_{1},p_{2})$, where $p_{i} : L^{n} \rightarrow L$ is a projection map with each $i = 1,2$ (see Theorem 2.24 of \cite{BoratPamukVergili:2023}). Thus, by using the higher contiguity distance, we have the alternative definition of the higher discrete topological complexity as follows:

\begin{theorem}\label{thm1}
	The higher ($n-$th) discrete topological complexity TC$_{n}(L)$ of a simplicial complex $L$ is SD$(p_{1},p_{2}\cdots,p_{n})$.
\end{theorem}

\quad The proof of Theorem \ref{thm1} is given in Theorem 2.1 of \cite{alaborciherdal:2023}. When $n=2$, TC$_{2}(L)$ corresponds to TC$(L)$. Also, by considering the quick higher SD-observation, we easily have TC$_{n}(L) \leq$ TC$_{n+1}(L)$. Note that this result is first proved in \cite{alaborciherdal:2023} (see Theorem 2.1 for the details of proof).

\begin{theorem}
	TC$_{n}(L) =$ TC$_{n}(N)$ if $L \sim N$ (see also Theorem 2.3 of \cite{alaborciherdal:2023}).
\end{theorem}

\begin{proof}
	Let $\alpha : L \rightarrow N$ and $\beta : N \rightarrow L$ be simplicial maps such that $\alpha \circ \beta \sim 1_{N}$ and $\beta \circ \alpha \sim 1_{L}$. Then we have that $\beta^{n} \circ \alpha^{n} = 1_{N^{n}}$ and $\alpha^{n} \circ \beta^{n} \sim 1_{L^{n}}$, i.e., $L^{n} \sim N^{n}$. Consider the following diagram with respect to the contiguity: 
	$$\xymatrix{
		L^{n} \ar[r]^{p_{1},\cdots,p_{n}} &
		L \ar[d]^{\alpha} \\
		N^{n} \ar[r]_{p^{'}_{1},\cdots,p^{'}_{n}} \ar[u]^{\beta^{n}} & N.}$$
	This means that $\alpha \circ p_{i} \circ \beta^{n} \sim p^{'}_{i}$. Thus, by Corollary \ref{cor1}, we obtain $$\text{SD}(p_{1},\cdots,p_{n}) = \text{SD}(p^{'}_{1},\cdots,p^{'}_{n}),$$ which shows that TC$_{n}(L) =$ TC$_{n}(N)$.
\end{proof}

\quad We now want to define TC$(\varphi)$ in terms of the contiguity distance.

\begin{theorem}\label{teo1}
	Let $\varphi : L \rightarrow L^{'}$ be a surjective simplicial finite-fibration. Then $$\text{TC}(\varphi) = \text{SD}(\varphi \circ \pi_{1},\pi_{2})$$ for the projection maps $\pi_{1} : L \times L^{'} \rightarrow L$ and $\pi_{2} : L \times L^{'} \rightarrow L^{'}$.
\end{theorem}

\begin{proof}
	Since TC$(\varphi) =$ Sg$(\pi_{\varphi})$, we shall show that SD$(\varphi \circ \pi_{1},\pi_{2}) =$ Sg$(\pi_{\varphi})$. First, assume that Sg$(\pi_{\varphi}) = s$. Then $L \times L^{'}$ can be written as the union of subcomplexes $L_{0},L_{1},\cdots,L_{s}$ for which $\pi_{\varphi}$ admits a simplicial map $\sigma_{k}$ for each $k = 0,\cdots,s$ with $\pi_{\varphi} \circ \sigma_{k} = 1_{L_{k}}$. For each $k$, we define a simplicial map $H_{k} : L_{k} \times I_{m} \rightarrow L^{'}$ by $H_{k}([x],[y],t) = \varphi(\sigma_{k}([x],[y])(t))$. Then $(\varphi \circ \pi_{1})\big|_{L_{k}}$ and $\pi_{2}\big|_{L_{k}}$ are in the same contiguity class, which shows that SD$(\varphi \circ \pi_{1},\pi_{2}) \leq s$. Conversely, assume that SD$(\varphi \circ \pi_{1},\pi_{2}) = s$. Then $(\varphi \circ \pi_{1})\big|_{L_{k}}$ and $\pi_{2}\big|_{L_{k}}$ are in the same contiguity class for each $k = 0,\cdots,s$, namely that, there is a simplicial map $H_{k} : L_{k} \times I_{m} \rightarrow L^{'}$ between $\varphi \circ \pi_{1}$ and $\pi_{2}$ for each $k$. Since $\varphi$ is a simplicial fibration, the commutative diagram
	$$\xymatrix{
		L_{k} \ar[r]^{\pi_{1}} \ar[d]_{incl} &
		L \ar[d]^{\varphi} \\
		L_{k} \times I_{m} \ar[r]_{H}  & L^{'}}$$
	admits a simplicial map $\tilde{H}_{k} : L_{k} \times I_{m} \rightarrow L$ such that $\varphi \circ \tilde{H}_{k} = H_{k}$ and $\tilde{H}_{k} \circ incl = \pi_{1}$. For each $k$, define a simplicial map $\sigma_{k} : L_{k} \rightarrow L^{I}$ by $\sigma_{k}([x],[y])(t) = \tilde{H}_{k}([x],[y],[t])$. Thus, we get $\pi_{\varphi} \circ \sigma_{k} = 1_{L_{k}}$, which shows that Sg$(\pi_{\varphi}) \leq s$.
\end{proof}

\section{Higher Discrete TC Of A Simplicial Fibration}
\label{sec:4}

\quad In \cite{iskaraca:2022} (see also \cite{aghimirebaba:2023}, the higher topological complexity of a surjective fibration is expressed in terms of the higher homotopic distance. Similarly, we can define the higher discrete topological complexity of a surjective simplicial fibration by using the higher contiguity distance as follows.

\begin{definition}\label{def2}
	Given a surjective simplicial finite-fibration $\varphi : L \rightarrow L^{'}$, the higher ($n-$th) discrete topological complexity of $\varphi$ is TC$_{n}(\varphi) =$ SD$(\varphi \circ p_{1},\cdots,\varphi \circ p_{n})$ for the projection $p_{i} : L^{n} \rightarrow L$ with each $i = 1,\cdots,n$.
\end{definition} 

\quad For any surjective simplicial finite-fibration $\varphi : L \rightarrow L^{'}$, we have that TC$_{2}(\varphi)$ in Definition \ref{def2}, coincides with TC$(\varphi)$ in Theorem \ref{teo1}. Indeed, by Corollary \ref{cor1} with considering the following commutative diagram, we find that $$\text{SD}(\varphi \circ p_{1},\varphi \circ p_{2}) = \text{SD}(\varphi \circ \pi_{1},\pi_{2})$$ for the projection maps $p_{i} : L^{2} \rightarrow L$ with each $i = 1,2$, $\pi_{1} : L \times L^{'} \rightarrow L$, and $\pi_{2} : L \times L^{'} \rightarrow L^{'}$.
$$\xymatrix{
	L \times L^{'} \ar[r]^{\varphi \circ \pi_{1}}_{\pi_{2}} &
	L^{'} \ar[d]^{\alpha = 1_{L^{'}}} \\
	L^{2} \ar[r]^{\varphi \circ p_{1}}_{\varphi \circ p_{2}} \ar[u]^{\beta = 1_{L} \times \varphi} & L^{'}.}$$
Note that $\alpha$ is clearly a strong equivalence, so it is enough to say that $\beta$ is also a strong equivalence. Since $\varphi$ is surjective, there is an element $[x^{'}] \in L^{'}$ such that $\varphi([x^{'}]) = [y]$. For a simplicial map $\omega : L \times L^{'} \rightarrow L \times L$ with $\omega([x],[y]) = ([x],[x^{'}])$, we get $\beta \circ \omega \sim 1_{L \times L^{'}}$ and $\omega \circ \beta \sim 1_{L^{2}}$. This shows that $\beta$ is a strong equivalence. Finally, we have the equality TC$_{2}(\varphi) =$ TC$(\varphi)$.

\begin{proposition}\label{prop3}
	Let $\varphi : L \rightarrow L^{'}$ be a surjective finite-fibration. Then
	 
	\textbf{i)} TC$_{n}(\varphi) \leq$ TC$_{n+1}(\varphi)$.
	
	\textbf{ii)} TC$_{n}(\varphi) =$ TC$_{n}(L)$ when $\varphi = 1_{L} : L \rightarrow L$.
	
	\textbf{iii)} TC$_{n}(\varphi) \leq$ TC$_{n}(L)$.
	
	\textbf{iv)} TC$(\varphi) \leq$ TC$_{n}(L)$.
	
	\textbf{v)} TC$_{n}(\varphi) = 0$ provided that $L$ or $L^{'}$ is strongly collapsible.
\end{proposition}

\begin{proof}
	\textbf{i)} It is clear from Definition \ref{def1}.
	
	\textbf{ii)} It follows from the fact that
	\begin{eqnarray*}
		\text{SD}(1_{L} \circ p_{1},\cdots,1_{L} \circ p_{n}) = SD(p_{1},\cdots,p_{n}).
	\end{eqnarray*}

    \textbf{iii)} The fact is a result of Proposition \ref{prop1} i).
    
    \textbf{iv)} By Proposition \ref{prop1} i), we get
    \begin{eqnarray*}
    	\text{SD}(\varphi \circ p_{1},\varphi \circ p_{2}) \leq \text{SD}(p_{1},p_{2}) \leq \text{SD}(p_{1},\cdots,p_{n}).
    \end{eqnarray*}

    \textbf{v)} If $L$ is strongly collapsible, then we get $1_{L} \sim c_{L}$. By using Proposition \ref{prop1} and Proposition \ref{prop2} i), we get
    \begin{eqnarray*}
    	\text{SD}(\varphi \circ p_{1},\cdots,\varphi \circ p_{n}) &=& \text{SD}(\varphi \circ 1_{L} \circ p_{1},\cdots,\varphi \circ 1_{L} \circ p_{n})\\ &=& \text{SD}(\varphi \circ c_{L} \circ p_{1},\cdots,\varphi \circ c_{L} \circ p_{n})\\ &=& \text{SD}(c^{'}_{L} \circ p_{1},\cdots,c^{'}_{L} \circ p_{n}),
    \end{eqnarray*}
    where $c^{'}_{L} = \varphi \circ c_{L}$ is a constant simplicial map. Since $c^{'}_{L} \circ p_{i} \sim c^{'}_{L} \circ p_{j}$ for any $i$, $j = 1,\cdots,n$, we conclude that $$\text{SD}(c^{'}_{L} \circ p_{1},\cdots,c^{'}_{L} \circ p_{n}) = 0.$$ Also, if $L^{'}$ is strongly collapsible, then we follow the same method starting with $$\text{SD}(\varphi \circ p_{1},\cdots,\varphi \circ p_{n}) = \text{SD}(1_{L^{'}} \circ \varphi \circ p_{1},\cdots,1_{L^{'}} \circ \varphi \circ p_{n})$$ by Proposition \ref{prop1} again.
\end{proof}

\begin{theorem}
	For a simplicial finite-fibration $\varphi : L \rightarrow N$, we have that $$\text{TC}_{n}(\varphi) \leq \min\{\text{TC}_{n}(L),\text{TC}_{n}(N)\}.$$
\end{theorem}

\begin{proof}
	It is enough to show that $$\text{TC}_{n}(\varphi) \leq \text{TC}_{n}(N)$$ from Proposition \ref{prop3} iii). Assume that $p_{i} : L^{n} \rightarrow L$ and $q_{i} : N^{n} \rightarrow N$ are projection maps for each $i = 1,\cdots,n$. Then $\varphi \circ p_{i} = q_{i} \circ \varphi^{n}$. Indeed,
	\begin{eqnarray*}
		\varphi \circ p_{i}([x_{1}],\cdots,[x_{n}]) = \varphi([x_{i}]) &=& [x^{'}_{i}] \\ 
		&=& q_{i}([x^{'}_{1}],\cdots,[x^{'}_{n}])\\ 
		&=& q_{i} \circ \varphi^{n}([x_{1}],\cdots,[x_{n}])
	\end{eqnarray*}
    for any $[x_{1}],\cdots,[x_{n}] \in L$ and $[x^{'}_{1}],\cdots,[x^{'}_{n}] \in N$. This concludes that
    \begin{eqnarray*}
    	\text{SD}(\varphi \circ p_{1},\cdots,\varphi \circ p_{n}) = \text{SD}(q_{1} \circ \varphi^{n},\cdots,q_{n} \circ \varphi^{n}) \leq \text{SD}(q_{1},\cdots,q_{n}).
    \end{eqnarray*}
\end{proof}

\begin{theorem}
	For any surjective simplicial finite-fibrations $\varphi : L \rightarrow N$ and \linebreak$\psi : N \rightarrow K$, we have that $$\text{TC}_{n}(\psi \circ \varphi) \leq \min\{\text{TC}_{n}(\varphi),\text{TC}_{n}(\psi)\}.$$
\end{theorem}

\begin{proof}
	Let $p_{i} : L^{n} \rightarrow L$ and $q_{i} : N^{n} \rightarrow N$ be the projection maps with each $i = 1,\cdots,n$. Then $\varphi \circ p_{i} = q_{i} \circ (\varphi,\cdots,\varphi)$. Indeed,
	\begin{eqnarray*}
		\varphi \circ p_{i}([x_{1}],\cdots,[x_{n}]) &=& \varphi([x_{i}])\\ 
		&=& q_{i}([\varphi(x_{1})],\cdots,[\varphi(x_{n})])\\ 
		&=& q_{i} \circ (\varphi,\cdots,\varphi)([x_{1}],\cdots,[x_{n}])
	\end{eqnarray*}
	for any $[x_{1}],\cdots,[x_{n}] \in L$ and $[x^{'}_{1}],\cdots,[x^{'}_{n}] \in N$. It follows that
	\begin{eqnarray*}
		\text{SD}(\psi \circ \varphi \circ p_{1},\cdots,\psi \circ \varphi \circ p_{n}) \leq \text{SD}(\varphi \circ p_{1},\cdots,\varphi \circ p_{n}),
	\end{eqnarray*}
    and
    \begin{eqnarray*}
    	\text{SD}(\psi \circ \varphi \circ p_{1},\cdots,\psi \circ \varphi \circ p_{n}) &=& \text{SD}(\psi \circ q_{1} \circ (\varphi,\cdots,\varphi),\cdots,\psi \circ q_{n} \circ (\varphi,\cdots,\varphi))\\
    	&\leq& \text{SD}(\psi \circ q_{1},\cdots,\psi \circ q_{n}),
    \end{eqnarray*}
    which conclude that TC$_{n}(\psi \circ \varphi) \leq$ TC$_{n}(\varphi)$ and TC$_{n}(\psi \circ \varphi) \leq$ TC$_{n}(\psi)$, respectively.
\end{proof}

\begin{corollary}
	Given any surjective simplicial finite-fibration $\varphi : L \rightarrow L^{'}$,
	
	\textbf{i)} TC$_{n}(\varphi) =$ TC$_{n}(L^{'})$ when $\varphi$ admits a right strong equivalence.
	
	\textbf{ii)} TC$_{n}(\varphi) =$ TC$_{n}(L)$ when $\varphi$ admits a left strong equivalence.
	
	\textbf{ii)} TC$_{n}(\varphi) =$ TC$_{n}(L) =$ TC$_{n}(L^{'})$ when $\varphi$ admits a strong equivalence.
\end{corollary}

\begin{proof}
	\textbf{i)} Let $\omega : L^{'} \rightarrow L$ be the right strong equivalence of $\varphi$, i.e., $\varphi \circ \omega \sim 1_{L^{'}}$. Then we find
	\begin{eqnarray*}
		\text{TC}_{n}(L^{'}) = \text{TC}_{n}(1_{L^{'}}) = \text{TC}_{n}(\varphi \circ \omega) \leq \text{TC}_{n}(\varphi) \leq \text{TC}_{n}(L^{'}).
	\end{eqnarray*}
    \textbf{ii)} Let $\gamma : L^{'} \rightarrow L$ be the left strong equivalence of $\varphi$, i.e., $\gamma \circ \varphi \sim 1_{L}$. Then we find
    \begin{eqnarray*}
    	\text{TC}_{n}(L) = \text{TC}_{n}(1_{L}) = \text{TC}_{n}(\gamma \circ \varphi) \leq \text{TC}_{n}(\varphi) \leq \text{TC}_{n}(L).
    \end{eqnarray*}
    \textbf{iii)} The result is the direct consequence of the first two parts.
\end{proof}

\section{scat-Related Results}
\label{sec:5}

\quad We have some results between scat and TC of a surjective simplicial fibration $\varphi_{1} : L \rightarrow L^{'}$.

\begin{proposition}\label{prop4}
	scat$(\varphi) \leq$ TC$(\varphi)$ for a given surjective simplicial finite-fibration $\varphi : L \rightarrow L^{'}$.
\end{proposition}

\begin{proof}
	Let $c_{v_{0}}$ be the constant simplicial map on $L$ at the point $v_{0}$, $p_{i}$ the projection simplicial map on $L$ with each $i = 1,2$, and $i_{1} : L \rightarrow L^{2}$ a simplicial map defined by $i_{1}(\sigma) = (\sigma,v_{0})$. Then we have that
	\begin{eqnarray*}
		\text{TC}(\varphi) &=& \text{SD}(\varphi \circ p_{1},\varphi \circ p_{2}) \geq \text{SD}(\varphi \circ p_{1} \circ i_{1},\varphi \circ p_{2} \circ i_{1})\\ 
		&=& \text{SD}(\varphi \circ 1_{L},\varphi \circ c_{v_{0}}) = \text{SD}(\varphi,\varphi \circ c_{v_{0}})\\ 
		&=& \text{scat}(\varphi).
	\end{eqnarray*}
\end{proof}

\begin{proposition}\label{prop5}
	scat$(L) \leq$ TC$(\varphi)$ for a given bijective simplicial finite-fibration $\varphi : L \rightarrow L^{'}$.
\end{proposition}

\begin{proof}
	Let $c_{v_{0}}$ be the constant simplicial map on $L$ at the point $v_{0}$, $p_{i}$ the projection simplicial map on $L$ with each $i = 1,2$, $i_{1} : L \rightarrow L^{2}$ a simplicial map defined by $i_{1}(\sigma) = (\sigma,v_{0})$, and $i_{2} : L \rightarrow L^{2}$ a simplicial map defined by $i_{1}(\sigma) = (v_{0},\sigma)$. Then we have that $p_{2} \circ i_{1} = c_{v_{0}} = p_{1} \circ i_{2}$. Since $\varphi$ is injective, there exists a simplicial map $\omega : L^{'} \rightarrow L$ with $\omega \circ \varphi = 1_{L}$. Moreover, we get
	\begin{eqnarray*}
		\text{TC}(\varphi) &=& \text{SD}(\varphi \circ p_{1},\varphi \circ p_{2}) \geq \text{SD}((\omega \circ \varphi) \circ (p_{1} \circ i_{1}),(\omega \circ \varphi) \circ (p_{2} \circ i_{1}))\\ 
		&=& \text{SD}(1_{L} \circ 1_{L},1_{L} \circ c_{v_{0}}) \\ 
		&=& \text{SD}(1_{L},c_{v_{0}}) = \text{scat}(L).
	\end{eqnarray*}
\end{proof}

\begin{proposition}\label{prop6}
	scat$(\varphi) \leq$ scat$(L)$ for a simplicial finite-fibration $\varphi : L \rightarrow L^{'}$.
\end{proposition}

\begin{proof}
	Let $c_{v_{0}}$ be the constant simplicial map on $L$ at the point $v_{0}$, $p_{i}$ the projection simplicial map on $L$ with each $i = 1,2$, $i_{1} : L \rightarrow L^{2}$ a simplicial map defined by $i_{1}(\sigma) = (\sigma,v_{0})$, and $i_{2} : L \rightarrow L^{2}$ a simplicial map defined by $i_{1}(\sigma) = (v_{0},\sigma)$. Then we find
	\begin{eqnarray*}
		\text{scat}(\varphi) &=& \text{SD}(\varphi \circ 1_{L},\varphi \circ c_{v_{0}})
		= \text{SD}(\varphi \circ p_{1} \circ i_{1},\varphi \circ p_{1} \circ i_{2})\\
		&\leq& \text{SD}(i_{1},i_{2}) = \text{scat}(L).
	\end{eqnarray*}
\end{proof}

\begin{theorem}\label{thm2}
	Given a bijective simplicial finite-fibration $\varphi : L \rightarrow L^{'}$, we have
	\begin{eqnarray*}
		\text{scat}(\varphi) \leq \text{scat}(L) \leq \text{TC}(\varphi) \leq \min\{\text{TC}(L),\text{TC}_{n}(\varphi)\} \leq \text{TC}_{n}(L).
	\end{eqnarray*}
\end{theorem}

\begin{proof}
	The result comes from Proposition \ref{prop6}, Proposition \ref{prop5}, and Proposition \ref{prop3} iv), respectively.
\end{proof}

\quad Combining with Corollary 2.27 of \cite{BoratPamukVergili:2023}, Theorem \ref{thm2} concludes the following result:

\begin{corollary}
	Given a bijective simplicial finite-fibration $\varphi : L \rightarrow L^{'}$, we have
	\begin{eqnarray*}
		\text{scat}(\varphi) \leq \text{scat}(L) \leq \text{TC}(\varphi) \leq \text{TC}(L) \leq \text{scat}(L^{2}).
	\end{eqnarray*}
\end{corollary}

\section{Conclusion}
\label{conc:5}

We make significant strides in understanding the discrete topological complexity (TC) of surjective fibrations, as well as exploring related concepts such as the higher contiguity distance between simplicial maps and the higher discrete TC number. By rigorously computing the TC number of surjective fibrations and investigating their relationship with other topological measures such as scat, we uncover valuable insights into the computational and structural properties of these mappings. Our findings not only contribute to the theoretical understanding of topological complexity but also have practical implications in fields such as robotics, computational biology, and geometric modeling. The insights gained from our study can inform the design of efficient algorithms for motion planning, aid in the analysis of complex biological systems, and enhance computational representations of geometric structures.

Various versions of TC numbers exist in topological spaces, as in the case of higher topological complexity TC$_{n}$. Some of these are monoidal topological complexity, symmetric topological complexity, parametrized topological complexity, mixed topological complexity, and relative topological complexity. The computation of each of the versions of such numbers on the simplicial complexes can be considered an open problem. In addition, concepts such as barycentric subdivision, which belong to the simplicial complex theory, can also be examined in the continuation of this study.

\end{document}